\documentclass{amsart}
\usepackage{amsfonts,amssymb,amsbsy,amsmath,amsthm,enumerate,verbatim,physics}
\usepackage{color}
\usepackage[colorlinks]{hyperref}
\usepackage{accents}

\newtheorem{theorem}{Theorem}[section]
\newtheorem{proposition}[theorem]{Proposition}
\newtheorem{lemma}[theorem]{Lemma}
\newtheorem{follow}[theorem]{Corollary}

\theoremstyle{definition}
\newtheorem{remark}[theorem]{Remark}

\newcommand{\bel}{\begin{equation} \label}
\newcommand{\ee}{\end{equation}}

\newcommand{\pd}{\partial}

\newcommand{\supp}{{\text{supp}}}
\newcommand{\dist}{{\text{dist}}}

\newcommand{\C}{{\mathbb C}}
\newcommand{\R}{{\mathbb R}}
\newcommand{\N}{{\mathbb N}}

\newcommand{\nap}{{\nabla^\prime}}

\newcommand{\sgn}{{\rm sgn}}
\newcommand{\cA}{{\mathcal A}}
\newcommand{\cC}{{\mathcal C}}

\newcommand{\cF}{{\mathcal F}}

\newcommand{\cO}{{\mathcal O}}

\newcommand{\cS}{{\mathcal S}}

\newcommand{\ma}{{\mathfrak{a}}}
\newcommand{\mc}{{\mathfrak{c}}}

\newcommand{\mn}{{\mathfrak{n}}}

\newcommand{\xp}{x^\prime}
\newcommand{\yp}{y^\prime}
\newcommand{\teap}{\tea^\prime}
\newcommand{\xip}{\xi^\prime}
\newcommand{\al}{{\alpha}}
\newcommand{\be}{{\beta}}
\newcommand{\de}{{\delta}}
\newcommand{\eps}{{\epsilon}}
\newcommand{\ga}{{\gamma}}
\newcommand{\Ga}{{\Gamma}}
\newcommand{\om}{{\omega}}
\newcommand{\Om}{{\Omega}}
\newcommand{\la}{{\lambda}}
\newcommand{\La}{{\Lambda}}
\newcommand{\tea}{{\theta}}
\newcommand{\mfe}{{\mathfrak{e}}}
\newcommand{\na}{{\nabla}}

\def\beq{\begin{equation}}
\def\eeq{\end{equation}}
\newcommand{\bea}{\begin{eqnarray}}
\newcommand{\eea}{\end{eqnarray}}
\newcommand{\beas}{\begin{eqnarray*}}
\newcommand{\eeas}{\end{eqnarray*}}

\numberwithin{equation}{section}

\title[Determination of the Schr\"odinger-Robin operator]{Determination of the Schr\"odinger-Robin operator by incomplete or asymptotic spectral boundary data}

\author[Mourad Choulli]{Mourad Choulli}
\address{Universit\'e de Lorraine}
\email{mourad.choulli@univ-lorraine.fr}

\author[Abdelmalek Metidji]{Abdelmalek Metidji}
\address{Aix Marseille Univ, Universit\'e de Toulon, CNRS, CPT, Marseille, France}
\email{abdelmalek.metidji1@gmail.com}

\author[\'Eric Soccorsi]{\'Eric Soccorsi}
\address{Aix Marseille Univ, Universit\'e de Toulon, CNRS, CPT, Marseille, France}
\email{eric.soccorsi@univ-amu.fr}

\begin{document}
\begin{abstract}
This article deals with the inverse problem of determining the unbounded real-valued electric potential of the Robin Laplacian on a bounded domain of dimension 3 or greater, by incomplete knowledge of its boundary spectral data. Namely, the main result establishes that the unknown potential can be H\"older stably retrieved from the asymptotic behavior of the eigenvalues and the sequence of the boundary measurements of the corresponding eigenfunctions where finitely many terms are missing.
\end{abstract}



\maketitle

\tableofcontents

\section{Introduction}
\label{sec-intro}

In this article, $\Om$ is a bounded domain of $\R^d$, $d \ge 3$, with $\cC^{1,1}$ boundary $\Ga$. Given $q \in L^{\sigma}(\Om,\R)$, where
$$\sigma =\sigma(d):=\max(2,d \slash 2),$$ 
we consider the Schr\"odinger operator
$$ \cA:=-\Delta + q, $$
endowed with the Robin boundary condition (RBC):
$$ \partial_\nu u + \al u =0\ \mbox{on}\ \Ga. $$
Here, $\partial_\nu$ is the normal derivative to $\Ga$ and $\al \in L^s(\Ga,\R)$, $s \in (d-1,\infty)$.  

While direct spectral analysis of the operator $\cA$ aims to obtain spectral information about $\cA$ from $q$, in this paper we rather focus on the inverse problem of recovering $q$ from knowledge of the so-called boundary spectral data of $\cA$.

\subsection{Spectral decomposition}
\label{sec-defA}

Assume that  $\al$ fulfills
\bel{i1} 
\al(x) \ge -\mc,\quad x \in \Ga,
\ee
for some constant $\mc \in (0,\mn^2)$, where $\mn$ stands for the norm of the trace operator 
$$v \in V:=H^1(\Om) \mapsto v_{| \Ga} \in L^2(\Ga).$$ 
Then, for all
$q \in L^{\sigma}(\Om,\R)$, the sesquilinear form
$$ \ma(u,v) := \int_{\Om} (\na u \cdot \na \overline{v} + q u \overline{v}) \dd x + \int_{\Ga} \alpha u \overline{v} \dd s,\quad u, v \in V, $$
is continuous on $V$ and $V$-coercive with respect to $H:=L^2(\Om)$ (see, e.g., \cite[Section 1.1]{CMS}).
As a consequence, the bounded operator $A : V \to V^*$, defined by
$$ \langle A u , v \rangle_{V^*,V} = \ma(u,v),\quad  u, v \in V,  $$
is selfadjoint and coercive. Here and below, $V^*$ denotes the dual space of $V$ and 
$\langle \cdot , \cdot \rangle_{V^*,V}$ stands for the duality pairing between $V$ and $V^*$.

Further, the spectrum of $A$ being discrete according to \cite[Theorem 2.37]{McL}, its eigenvalues arranged in non-decreasing order and repeated with the multiplicity are written
$$-\infty < \la_1\le \la_2 \le \ldots \le \la_n \le \ldots $$
and we recall that $\lim_{n \to \infty} \la_n = + \infty$.
Moreover, we denote by $\{ \phi_n,\ n \in \N \}$, $\N:= \{1,2,\ldots \}$, an orthonormal basis of $H$ made of eigenfunctions $\phi_n \in V$ of $A$, such that
\bel{i2} 
\ma(\phi_n,v) = \la_n ( \phi_n , v )_H,\quad v \in V,
\ee
where $(\cdot,\cdot)_H$ designates the usual scalar product in $H$. For the sake of brevity, we write
$$ \psi_n := {\phi_n}_{| \Ga},\quad n \in \N, $$
in what follows.

The aim of this paper is twofold. First, to show that $q$ is stably determined from the asymptotic knowledge of $(\la_n)$ at infinity and the full knowledge of $(\psi_n)$. Second, that this result remains valid under more stringent conditions on $q$, $\al$ and $\Ga$, when finitely many boundary measurements $\psi_n$ are unknown. As the full boundary spectral data (BSD) associated with $q$ (or $A$) are commonly defined (see, e.g., \cite{CMS,NSU,Sm}) by
$$ \mathrm{BSD}(q):= \{ (\la_n,\psi_n),\ n \in \N \}, $$
the spectral data used in this article will be referred as asymptotic-full in the first case, and asymptotic-incomplete in the second one.

\subsection{Main results}

We use the same notations as in Subsection \ref{sec-defA}. Namely, for all $n \in \N$, we denote by $\la_n$, $\phi_n$ and $\psi_n$, the $n$-th eigenvalue, eigenfunction and corresponding Dirichlet-trace, respectively, of the operator $A$ defined in Subsection
\ref{sec-defA}. Similarly, we write $(\tilde{\la}_n,\tilde{\phi}_n,\tilde{\psi}_n)$ instead of $(\la_n,\phi_n,\psi_n)$ when the potential $q$ is replaced by $\tilde{q}$. Define

$$Q_M := \{ q \in L^{\sigma}(\Om,\R),\ \norm{q}_{L^{\sigma}(\Om)} \le M \}, $$
where $M \in (0,\infty)$ is arbitrarily fixed. Our main result in the asymptotic-full case is as follows. 

\begin{theorem}
\label{thm1}
Pick $q,\tilde{q}\in Q_M$ and assume that $ (\la_n-\tilde{\la}_n) \in \ell^\infty$ and that $\psi_n=\tilde{\psi}_n$ for all $n \in \N$.
Then, there exists a positive constant $C$, depending only on $\Om$ and $M$, such that
\bel{se} 
\norm{q-\tilde{q}}_{H^{-1}(\Om)} \le C \limsup_{n \to \infty} 
\abs{\la_n-\tilde{\la}_n}^\frac{2}{d+2}.
\ee
\end{theorem}

Notice that Theorem \ref{thm1} is comparable to \cite[Theorem 1.1]{CMS} but cannot be deduced from it, as the stability inequality \eqref{se} involves only the asymptotic distance between the eigenvalues. Furthermore, the following uniqueness result is a direct byproduct of Theorem \ref{thm1}.

\begin{follow}
\label{cor1}
Let $q, \tilde{q}\in L^{\sigma}(\Om)$. If $\displaystyle \lim_{n \to \infty} (\la_n-\tilde{\la}_n)=0$ and $\psi_n=\tilde{\psi}_n$ for all $\in \N$,
then $ q=\tilde{q}$.
\end{follow}

Corollary \ref{cor1} is a slight improvement of \cite[Theorem 1.2]{CMS} in the sense that it shows that asymptotic knowledge of the eigenvalues only, as compared with the whole sequence without possibly a finite number of terms in \cite[Theorem 1.2]{CMS}, is needed to identify the zero-th order coefficient of the Schr\"odinger operator, when all boundary measurements of the eigenfunctions are known. 

As Theorem \ref{thm1} and Corollary \ref{cor1} hold for any $\cC^{1,1}$ bounded domain $\Om$, higher regularity assumption (at least $\cC^2$)  plus some suitable geometric condition, that we will make precise below, are requested on $\Om$ for the asymptotic-incomplete case. 
Actually, since $\Om$ is bounded, the boundary $\Ga$ is defined by a local chart made of a finite number $J \in \N$ of bounded open subsets $\cO_j \subset \R^d$,
$\om_j \subset \R^{d-1}$, and functions $\vartheta_j : \overline{\om_j}\rightarrow \R$ having the same regularity as $\Omega$, where $j=1,\ldots,J$, such that
\bel{loc-chart} 
\Ga \subset \bigcup_{j=1}^J \Ga_j\ \mbox{and}\
\Ga_j :=\Ga \cap \cO_j= \{ (\xp,\vartheta_j(\xp)),\ \xp \in \om_j \}.
\ee
Here and below we write $x=(\xp,x_d)$ and $\xp:=(x_1,\ldots,x_{d-1})$, for a generic point $x \in \R^d$. 
Therefore, we have $\Ga_j=f_j^{-1}(0)$ for all $j=1,\ldots,J$, where
$$ f_j(x):=\vartheta_j(\xp) - x_d,\quad x=(\xp,x_d) \in \om_j \times \R, $$
and consequently
\bel{b1} 
\nu_j(\xp) := \frac{\nabla f_j(\xp,\vartheta_j(\xp))}{\abs{\nabla f_j(\xp,\vartheta_j(\xp))}}
=\frac{\left( \nap \vartheta_j(\xp), -1\right)^T}{\sqrt{1+\abs{\nap \vartheta_j(\xp)}^2}},\quad \xp \in \om_j,
\ee
is the unit exterior normal vector to $\Ga_j$ at $x=(\xp,\vartheta_j(\xp))\in \Ga_j$. Herein the symbol $\abs{\cdot}$ denotes the Euclidian norm in either $\R^{d-1}$ or $\R^d$, and $\nap$ stands for $\nabla_{\xp}$, the gradient operator with respect to $x^\prime$.  
We shall impose a dual geometric assumption on $\Om$, requesting, first, that for all $\xi \in \mathbb{S}^{d-1}$ and all $j=1,\ldots,J$, the set 
\bel{gc1}
\cS_j(\xi):=\{ x=(\xp,\vartheta_j(\xp)) \in \Ga_j,\ \nu_j(\xp)\ \mbox{is parallel with}\ \xi \}\ \mbox{is at most finite}
\ee
and, second, that for any $x =(\xp,\vartheta_j(\xp)) \in \cS_j(\xi)$, the following Hessian matrix
\bel{gc2}
{\mathbb D^\prime}^2 \vartheta_j(\xp) := \left( \pd_{\xp_p,\xp_q}^2 \psi_j(\xp)\right)_{1 \le p,q \le d-1} 
\ee
is non singular. We point out that the conditions \eqref{gc1}-\eqref{gc2}, which might seem quite technical at first sight, are especially fulfilled when the domain $\Om$ is strongly convex (see Subection \ref{sec-comments} for more details on this peculiar point).

This being said, the stability result in the asymptotic-incomplete case can be stated as follows, where the symbol $\lceil s \rceil$ denotes the smallest integer greater or equal to $s\in (0,\infty)$.

\begin{theorem}
\label{thm2}
For $s \in \left(\frac{d-1}{2},\infty \right)$, put $m:=\lceil s \rceil + 1$, let the boundary $\Ga$ be $\cC^{m+2}$ and fulfill \eqref{gc1}-\eqref{gc2}, and let $\al \in \cC^{m,r}(\Ga)$ for some $r \in \left( \frac{1}{2},1 \right)$. Pick $q$ and $\tilde{q}$ in $W^{p,\infty}(\Om) \cap Q_M$ and assume that either of the two following conditions
\bel{c1}
q, \tilde{q} \ge 0 \quad \mbox{and}\quad  \al \ge c >0 ,
\ee
or
\bel{c2}
q,\tilde{q} \ge c>0 \quad \mbox{and}\quad \al \ge 0
\ee
is satisfied. Then, if $ (\la_n-\tilde{\la}_n) \in \ell^\infty$ and $\psi_n=\tilde{\psi}_n$ for all $n \ge n_0$, where $n_0 \in \N$ is arbitrary, we have
\[
\norm{q-\tilde{q}}_{H^{-1}(\Om)} \le C \limsup_{n \to \infty} 
\abs{\la_n-\tilde{\la}_n}^\frac{2}{d+2}
\]
for some constant $C>0$ depending only on $\Omega$ and $M$.
\end{theorem}

The novelty of the stability inequality of Theorem \ref{thm2}, as compared to the one stated  in \cite[Theorem 1.2]{CMS}, is that it does not involve any quantitative information about the difference $\psi_n-\tilde{\psi}_n$ for $n=1,\ldots,n_0-1$. Moreover, Theorem \ref{thm2} immediately yields the subsequent uniqueness result.

\begin{follow}
\label{cor2}
Let $\Om$, $\al$, $q$ and $\tilde{q}$ be as in Theorem \ref{thm2}. Assume that $\displaystyle \lim_{n \to \infty} (\la_n-\tilde{\la}_n)=0$ and that $\psi_n=\tilde{\psi}_n$ on $\Gamma$ for all $n\ge n_0$, where $n_0$ is arbitrary in $\N$. Then we have $q=\tilde{q}$.
\end{follow}

Another useful consequence of Theorem \ref{thm2} is the following identification result with incomplete partial boundary data, where the trace of the normal derivative of the eigenfunctions is taken only on an arbitrary open subset of $\Ga$.    

\begin{follow}
\label{cor3}
Let $\Om$, $\al$, $q$ and $\tilde{q}$ be as in Theorem \ref{thm2}. Assume that
$$
q = \tilde{q}\quad \mbox{in}\ \Om_0,
$$
where $\Om_0$ is a fixed neighborhood of $\Ga$ in $\overline{\Om}$. 
Then, given a non-empty open subset $\Ga_* \subset \Ga$ and $n_0 \in \N$, we have $q=\tilde{q}$ in $\Om$ whenever the identity 
$$(\la_n,{\phi_n}_{| \Ga_*}) =  (\tilde{\la}_n, {\tilde{\phi}_n}_{| \Ga_*})$$ 
holds for all $n \ge n_0$.
\end{follow}

\subsection{Remark on the special case of strongly convex domains}
\label{sec-comments}

The domain $\Om$ is said strongly convex if the local chart \eqref{loc-chart} of $\Ga$ is associated with strongly convex mappings $\vartheta_j$ on $\om_j$, $j=1,\ldots,J$, that is to say if there exists $\be_j \in (0,\infty)$ such that the function
$$ \om_j \ni \xp \mapsto \vartheta_j(\xp) - \frac{\be_j}{2} \abs{\xp}^2 $$
is convex. This can be equivalently reformulated as either
\bel{r-cvx1} 
(\nap \vartheta_j(\xp)-\nap \vartheta_j(\yp)) \cdot (\xp-\yp) \ge \be_j \abs{\xp-\yp}^2,\ \xp \in \om_j,\ \yp \in \om_j,
\ee
where the notation $\cdot$ stands for the Euclidian scalar product in $\R^{d-1}$,
or as
\bel{r-cvx2} 
{\mathbb D^\prime}^2 \vartheta_j(\xp) \yp \cdot \yp \ge \be_j \abs{\yp}^2,\ \xp \in \om_j,\ \yp \in \om_j.
\ee
It turns out that any strongly convex $\cC^2$ bounded domain $\Om$ satisfies the geometric condition \eqref{gc1}-\eqref{gc2}. As a matter of fact, for each fixed $\xi=(\xip,\xi_d)^T \in \mathbb{S}^{d-1}$ and all $j=1,\ldots,J$, the set $\cS_j(\xi)$ contains at most one point. Indeed, for any two points $\xp \in \cS_j(\xi)$ and $\yp \in \cS_j(\xi)$, we see from \eqref{b1} that $\xi_d \neq 0$ and
$$
\nap \vartheta_j(\xp)=\nap \vartheta_j(\yp)=-\frac{\xip}{\xi_d},
$$
which, combined with \eqref{r-cvx1}, yields that $\xp=\yp$. 
Moreover, it follows readily from \eqref{r-cvx2} that the matrix ${\mathbb D^\prime}^2 \vartheta_j(\xp)$ is positive definite, showing that $\Om$ fulfills \eqref{gc1}-\eqref{gc2}.

\subsection{Selected bibliography}

Mathematical inverse spectral theory dates back, at least, to the late 20s, when
Ambarzumian's paper \cite{A} established that the real-valued bounded potential $q$ of the periodic realization of the Sturm-Liouville operator $-\partial_x^2+q$ acting in $L^2(0,2\pi)$, is zero, if and only if its spectrum equals $\{ n^2,\ n \in \N \}$. For the same operator, but defined on $(0,\pi)$ and equipped with homogeneous Dirichlet boundary conditions (DBC), Borg \cite{B} and Levinson \cite{L} showed that, although the spectrum does not discriminate between symmetric potentials about the midpoint of $(0,\pi)$, additional data $\{ \norm{\phi_n}_{L^2(0,\pi)},\ n \in \N \}$, where $\{ \phi_n,\ n \in \N \}$ is a $L^2(0,\pi)$-orthogonal basis of eigenfunctions satisfying $\phi_n^\prime(0)=1$, uniquely determine $q$. Later on, substituting $\phi_n^\prime(\pi)$ by $\norm{\phi_n}_{L^2(0,\pi)}$, the same uniqueness result was derived by Gel'fand and Levitan in \cite{GL}.

The approaches from Borg, Levinson or Gel'fand and Levitan, are highly one-dimensional techniques that could not be adapted to the multi-dimensional case. But one great idea emerged in the 80s that paved the way to solving multi-dimension Borg-Levinson problems. In 1987, inspired by the scattering results of Fadeev \cite{F}, Silvester and Uhlmann introduced the so-called complex geometric optics (CGO) solutions in the context of inverse problems. Using these specifically designed perturbations of the exponential function, Nachman, Sylvester and Uhlmann \cite{NSU}, and Novikov \cite{N}, showed unique determination of the Dirichlet Laplacian in dimension 2 or greater, from its boundary spectral data (BSD), defined as the eigenpairs $(\la_n,\partial_\nu \phi_n)$, $n \in \N$, formed by the eigenvalues and the Neumann trace of the corresponding eigenfunctions. The stability issue for the same problem, was solved by Alessandrini and Sylvester in \cite{AS} (see also \cite[Theorem 2.31]{C1}). 
In 1991, Isozaki noticed that the identification result of \cite{NSU} is still valid by replacing the full set of BSD $\{ (\la_n,\partial_\nu \phi_n),\ n \in \N \}$ by an incomplete one $\{ (\la_n,\partial_\nu \phi_n),\ n \ge n_0 \}$ where $n_0 \in \N$ is arbitrary. From this point onwards, the focus was on downsizing the data. Following this trend, it was established in \cite{CS,KKS,S} that $q$ can be stably identified from the asymptotics of the BSD, i.e., that knowledge of
$\lim_{n \to \infty} (\la_n,\partial_\nu \phi_n)$ stably determines the Dirichlet Laplacian.

Notice that in the above quoted references, the unknown potential $q$ is assumed to be bounded (in fact assuming that $q\in L^d(\Omega,\R)$ is enough). But still there is a small number of mathematical papers dealing with inverse spectral problems with unbounded coefficients (precisely when $q\in L^p(\Omega,\R)$ with $d/2\le p<d$). The first one is \cite{PS} and was written by P\"aiv\"arinta and Serov. They showed unique determination of real-valued electric potentials in $L^r(\Om,\R)$, $r> d \slash 2$, from the full BSD. Later on, in \cite{P}, Pohjola retrieved unknown potentials in $L^{d \slash 2}(\Om,\R)$ from either full BSD when $d = 3$, or incomplete BSD when $d \neq 4,$ and of an unknown potential in $L^r(\Om,\R)$ with $r > d \slash 2$ and $d=3$, from incomplete BSD. Stability inequalities associated with Pohjola's uniqueness results \cite{P} were established by the first author in \cite{C4}.  More recently, it was established in \cite{BKMS} that, when $d \ge 3$ and $r=\max(2,3d/5)$, $q \in L^r(\Om,\R)$ is uniquely determined by the asymptotics of the BSD, while the stability issue for the same problem was treated in
\cite{KS}. 

The uniqueness of the determination of the Dirichlet-Laplace-Beltrami operator from its spectral boundary data was studied by Katchalov, Kurylev, and Lassas in \cite{KKL}. Various stability inequalities for this inverse spectral problem have recently been established by the first author and Yamamoto in \cite{CY}.

All the aforementioned results were derived for Schr\"odinger operators endowed with Dirichlet boundary conditions, except for \cite{NSU} where a bounded electric potential is determined by full BSD of the Robin Laplacian, defined as $(\lambda_n,\psi_n)$, $n \in \N$, where $\psi_n={\phi_n}_{| \Ga}$ stands for the Dirichlet trace of the eigenfunction $\phi_n$. Using a heuristic approach, Smirnova  asserted in \cite{Sm} that semi-incomplete BSD, i.e., full data BSD where a finite number of the eigenvalues remain unknown, uniquely  determine the bounded potential. This claimed was proved in \cite{CMS} and even extended to unbounded potentials $q \in L^r(\Om,\R)$, where $r=d \slash 2$ when $d \ge 4$, and $r>d \slash 2$ when $d=3$.

\subsection{Outline}

The article is organized as follows. In Section \ref{sec-pre} we derive the so-called Isozaki formula, that the derivation of Theorems \ref{thm1} and \ref{thm2}, and Corollary \ref{cor3}, is based on. The proof of these three results is displayed in Section \ref{sec-proof}. In particular, the proof of Theorem \ref{thm2} boils down to a Van-der-Corput type inequality that is established in Section \ref{sec-VdC}. Finally, in the Appendix, we gather several technical results that are needed by the proof of Theorem \ref{thm2}.

\section{Preliminaries}

\label{sec-pre}
Let $q \in L^{\sigma}(\Om,\R)$ and let $\lambda \in \rho(A)$, where $\rho(A) := \C \setminus \{ \la_n,\ n \in \N \}$ is the resolvent set of $A$. Pick $G \in H^{2+\eps}(\Om)$, where $\eps=0$ when $d \ne 4$, and $\eps$ is arbitrary in $(0,1)$ when $d=4$, in such a way that $q G \in L^2(\Om)$ by Sobolev imbedding theorem (see, e.g., \cite[Section 2.2]{CMS}). Then, putting
$g:=\partial_{\nu} G + \al G_{| \Ga}$, we recall for further use from \cite[Lemma 2.2]{CMS} that
\bel{p0} 
u_\lambda(g) := (A-\la)^{-1}(\Delta -q+\la)G + G,
\ee
is the unique $H^1(\Om)$-solution of the following BVP (boundary value problem)
\bel{go0} 
\left\{ \begin{array}{ll} (-\Delta+q-\la)u =0 & \mbox{in}\ \Om, \\ \pd_\nu u + u_{| \Ga}=g & \mbox{on}\ \Ga. \end{array} \right.
\ee

\subsection{Geometric optic solutions}
For $\om \in \mathbb{S}^{d-1}$ and $\la \in \rho(A)$, we introduce
$$ \mfe_{\la,\om}(x):= e^{i \sqrt{\la}\ \om \cdot x}, $$
and we set
\bel{go1}
G:=\mfe_{\la,\om},\quad g:=\partial_{\nu} G + \al G_{| \Ga}.
\ee
Similarly, for $\tea \in \mathbb{S}^{d-1}$ we put
\bel{go2} 
H:=\mfe_{\la,\tea},\quad h:=\partial_{\nu} H + \al H_{| \Ga}
\ee
and define
\bel{go3}
S(\la,\om,\tea) := (u_\la , h )_{L^2(\Ga)} = \int_{\Ga} u_\la \overline{h}\ \dd s,
\ee
where $u_\la:=u_\la(g)$ is given by \eqref{p0}. Now, using that $u_\la$ solves \eqref{go0}, we get upon arguing as in the derivation of \cite[Eq. (3.6)-(3.7)]{CMS}, that
\bel{go3b}
S(\la,\om,\tea) = \int_{\Ga} (i \sqrt{\la} \om \cdot \nu + \al) \mfe_{\la,\om-\tea}\ \dd s + \upsilon(\la,\om,\tea),
\ee
where
%\bel{go4}
\begin{align}
&\upsilon(\la,\om,\tea):=r(\la,\om,\tea) - \int_{\Om} q \mfe_{\la,\om-\tea}\ \dd x, \label{go4}
\\
&r(\la,\om,\tea):= \int_{\Om} q \mfe_{\la,-\tea} (A-\la)^{-1} (q \mfe_{\la,\om})\ \dd x. \label{go4.1}
\end{align}
%\ee
Next, for $\tilde{q} \in L^{\sigma}(\Om)$, we denote by $\tilde{A}$ the operator defined in Subsection \ref{sec-defA} upon replacing $q$ by $\tilde{q}$. Then, we pick $\la \in \rho(A) \cap \rho(\tilde{A})$, write $\tilde{u}_\la$ instead of $\tilde{u}_\la(g)$, and define 
$\tilde{S}(\la,\om,\tea)$, $\tilde{\upsilon}(\la,\om,\tea)$ and $\tilde{r}(\la,\om,\tea)$ by substituting $(\tilde{q},\tilde{A},\tilde{u}_\la)$ for $(q,A,u_\la)$ in \eqref{go3} and in \eqref{go4}. Setting 
%\bel{go5}
\begin{align}
&S_*(\la,\om,\tea):=\tilde{S}(\la,\om,\tea)-S(\la,\om,\tea), \label{go5}
\\
& r_*(\la,\om,\tea):=\tilde{r}(\la,\om,\tea)-r(\la,\om,\tea), \label{go5.1}
\end{align}
%\ee
we then deduce from  \eqref{go3b}-\eqref{go5.1} that
\bel{go6} 
\int_\Om \varrho \mfe_{\la,\om-\tea}\ \dd x = S_*(\la,\om,\tea) - r_*(\la,\om,\tea),
\ee
where
\bel{go7}
\varrho := (\tilde{q}-q) \chi_{\Om},
\ee
the notation $\chi_\Om$ standing for the characteristic function of $\Om$.

\subsection{Isozaki's asymptotic formula}

Fix $\xi \in \R^d$, pick $\eta \in \mathbb{S}^{d-1}$ such that
$$ \xi \cdot \eta =0, $$
and for $\tau \in \left( \frac{\abs{\xi}}{2},\infty \right)$, plug $\la=\la_\tau:= (\tau+i)^2$,
\begin{align}
%\bel{if0}
&\tea=\tea_\tau:=\left( 1- \frac{\abs{\xi}}{4 \tau^2} \right)^{\frac{1}{2}} \eta - \frac{\xi}{2 \tau},\label{if0}
\\
& \om=\om_\tau:= \left( 1- \frac{\abs{\xi}}{4 \tau^2} \right)^{\frac{1}{2}} \eta + \frac{\xi}{2 \tau} \label{if0.1}
%\ee
\end{align}
into \eqref{go6}. We obtain that
\bel{if1}
\hat{\varrho}\left( \left( 1 + \frac{i}{\tau} \right) \xi \right) = S_*(\la_\tau,\om_\tau,\tea_\tau) - r_*(\la_\tau,\om_\tau,\tea_\tau),
\ee
where $\hat{\varrho}$ denotes the Fourier-Laplace transform of $\varrho$:
\[
\hat{\varrho}(z):=\int_{\mathbb{R}^d}e^{-i x\cdot z}\varrho(x)dx,\quad z\in \mathbb{C}^d.
\] 
Note that, since $\varrho$ has  compact support, $\hat{\varrho}$ is an entire function.

Further, in light of \cite[Eq. (3.14)]{CMS}, we have
$$ \lim_{\tau \to \infty} r(\la_\tau,\om_\tau,\tea_\tau)=\lim_{\tau \to \infty} \tilde{r}(\la_\tau,\om_\tau,\tea_\tau)=0. $$
This combined with \eqref{go5}, \eqref{if1} and the continuity of $\hat{\varrho}$, then yields that
\bel{if2}
\hat{\varrho}( \xi ) = \lim_{\tau \to \infty} S_*(\la_\tau,\om_\tau,\tea_\tau).
\ee
The next step is to express the right-hand side of \eqref{if2} in terms of the BSD. This will be done by relating 
$S_*(\la_\tau,\om_\tau,\tea_\tau)$ for fixed $\tau \in (0,\infty)$, to the BSD.

%\subsection{Linking $S_*(\la_\tau,\om_\tau,\tea_\tau)$ to the BSD}
\subsection{A representation formula}
Let us start by recalling from \cite[Lemma 2.3]{CMS} that for $\mu \in \rho(A)$, we have
$$ {u_{\la_\tau}}_{| \Ga} = {u_\mu}_{| \Ga} + (\la-\mu) \sum_{n \ge 1} \frac{(g , \psi_n)_{L^2(\Ga)}}{(\la_n-\la_\tau)(\la_n-\mu)}\psi_n,\quad \tau \in (0,\infty),$$
the series being convergent in $H^{\frac{1}{2}}(\Ga)$. 
Putting this together with \eqref{go3} and \eqref{go5}, we get for all $\tau \in (0,\infty)$ that
\bel{s1}
S_*(\la_\tau,\om_\tau,\tea_\tau) = S_*(\mu,\om_\tau,\tea_\tau) + \ell_\tau(\mu) + \kappa_\tau(\mu),\quad \mu \in \rho(A) \cap \rho(\tilde{A}),
\ee
where we have set
$$
\ell_{\tau}(\mu):=(\la_\tau-\mu) \sum_{n=1}^\infty \frac{d_n-\tilde{d}_n}{(\la_n-\la_\tau)(\la_n-\mu)},
$$
$$
\kappa_{\tau}(\mu):= (\la_\tau-\mu)\sum_{n=1}^\infty \left( \frac{1}{(\la_n-\la_\tau)(\la_n-\mu)}- \frac{1}{(\tilde{\la}_n-\la_\tau)(\tilde{\la}_n-\mu)} \right) \tilde{d}_n,
$$
and
\bel{s2} 
d_n:= (g , \psi_n ) (\psi_n, h ),\quad \tilde{d}_n:= (g , \tilde{\psi}_n ) (\tilde{\psi}_n, h ),\quad n \in \N.
\ee
Notice that \eqref{s1} holds for all $\mu \in (-\infty,\ \min(\la_1,\tilde{\la}_1))$, and that 
\bel{s3}
\lim_{\mu \to -\infty} S_*(\mu,\om_\tau,\tea_\tau) =0,\quad \tau \in (0,\infty),
\ee
according to \cite[Lemma 2.5]{CMS}. Moreover, keeping in mind that $(\tilde{\la}_n-\la_n) \in \ell^\infty$, we have
\bel{s4}
\lim_{\mu \to -\infty} \ell_{\tau}(\mu) = \sum_{n=1}^\infty \frac{d_n-\tilde{d}_n}{\la_n-\la_\tau},\ \tau \in (0,\infty),
\ee
from \cite[Eq. (3.22)]{CMS}, and
\bel{s5}
\lim_{\mu \to -\infty} \kappa_{\tau}(\mu) = \sum_{n=1}^\infty \frac{\tilde{\la}_n-\la_n}{(\la_n-\la_\tau)(\tilde{\la}_n-\la_\tau)} \tilde{d}_n,\quad \tau \in (0,\infty),
\ee
according to \cite[Eq. (3.24)]{CMS}. Therefore, sending $\mu$ to $-\infty$ in \eqref{s1}, we obtain that
$$ 
S_*(\la_\tau,\om_\tau,\tea_\tau) = \sum_{n=1}^\infty \frac{d_n-\tilde{d}_n}{\la_n-\la_\tau} +  \sum_{n=1}^\infty \frac{\tilde{\la}_n-\la_n}{(\la_n-\la_\tau)(\tilde{\la}_n-\la_\tau)} \tilde{d}_n,\quad  \tau \in (0,\infty),
$$
according to \eqref{s3}--\eqref{s5}. In view of \eqref{if2}, this implies
\bel{s6}
\hat{\varrho}( \xi ) = \lim_{\tau \to \infty} \left( \sum_{n=1}^\infty \frac{d_n-\tilde{d}_n}{\la_n-\la_\tau} +  \sum_{n=1}^\infty \frac{\tilde{\la}_n-\la_n}{(\la_n-\la_\tau)(\tilde{\la}_n-\la_\tau)} \tilde{d}_n \right),
\ee
whenever $(\tilde{\la}_n-\la_n) \in \ell^\infty$.

\section{Proof of Theorems \ref{thm1} and \ref{thm2}, and Corollary \ref{cor3}}
\label{sec-proof}

\subsection{Proof of Theorem \ref{thm1}}
\label{sec-prthm1}
We break the proof into two main steps.\\

\noindent {\it Step 1.} We show that
\bel{dm0}
\abs{\hat{\varrho}(\xi)} \le C \limsup_{n \to \infty} \abs{\la_n-\tilde{\la}_n}. 
\ee
For this purpose, we notice that, since $d_n=\tilde{d}_n$ for all $n \in \N$ by assumption, we have
\bel{dm1}
\hat{\varrho}(\xi) = \lim_{\tau \to \infty}
\sum_{n=1}^\infty \frac{\tilde{\la}_n-\la_n}{(\la_n-\la_\tau)(\tilde{\la}_n-\la_\tau)} \tilde{d}_n,\quad \xi \in \R^d.
\ee
according to \eqref{s6}. 

Set $\La_1:=\norm{(\la_n-\tilde{\la}_n)}_{\ell^\infty}$. Then, bearing in mind that
for all $n \in \N$, 
$$ \abs{\frac{\la_\tau-\tilde{\la}_n}{\la_\tau-\la_n}} \le 1 + 
\frac{\La_1}{\abs{\Im \la_\tau}} \le 1 + \frac{\La_1}{2 \tau},\ \tau \in (0,\infty),$$
we have
\bea
\abs{\frac{\tilde{\la}_n-\la_n}{(\la_n-\la_\tau)(\tilde{\la}_n-\la_\tau)} \tilde{d}_n} & = & 
\abs{\frac{(\tilde{\psi}_n,g)}{\tilde{\la}_n-\la_\tau}} \abs{\frac{(\tilde{\psi}_n,h)}{\la_n-\la_\tau}} \abs{\la_n-\tilde{\la}_n}
\nonumber \\
& \le & \left( 1 + \frac{\La_1}{2 \tau} \right) \abs{\frac{(\tilde{\psi}_n,g)}{\tilde{\la}_n-\la_\tau}} \abs{\frac{(\tilde{\psi}_n,h)}{\tilde{\la}_n-\la_\tau}} \abs{\la_n-\tilde{\la}_n}. \label{dm2}
\eea
Then, using that 
$$ \sum_{n=1}^\infty \abs{\frac{(\tilde{\psi}_n,g)}{\tilde{\la}_n-\la_\tau}}^2 = \norm{u_{\la_\tau}(g)}_{L^2(\Om)}^2
$$
is bounded by a positive constant $C$ that is independent of $\tau$ and $\xi$, 
we get from \eqref{dm2} upon applying the Cauchy-Schwarz inequality, that for all $N \in \N$,
\bea
\sum_{n=N}^\infty \abs{\frac{\tilde{\la}_n-\la_n}{(\la_n-\la_\tau)(\tilde{\la}_n-\la_\tau)} \tilde{d}_n}
& \le & \left( 1 + \frac{\La_1}{2 \tau} \right) \norm{u_{\la_\tau}(g)}_{L^2(\Om)} \norm{u_{\la_\tau}(h)}_{L^2(\Om)} \La_N \nonumber \\
& \le &  C \left( 1 + \frac{\La_1}{2 \tau} \right) \La_N, \label{dm3}
\eea
where $\La_N:= \sup_{n \ge N} \abs{\la_n-\tilde{\la}_n}$.
Further, for all $n \in \N$, we have
\beas
\abs{\frac{\tilde{d}_n}{(\la_n-\la_\tau)(\tilde{\la}_n-\la_\tau)}} & \le & 
\frac{\norm{g}_{L^2(\Ga)} \norm{h}_{L^2(\Ga)} \norm{\tilde{\psi}_n}_{L^2(\Ga)}^2}
{\abs{(\tau+i)^2-\la_n} \abs{(\tau+i)^2-\tilde{\la}_n}} \\
& \le & \frac{C_n \tau^2}{\left( (\tau^2-\la_n-1)^2+4\tau^2 \right)^{\frac{1}{2}} \left( (\tau^2-\tilde{\la}_n-1)^2+4\tau^2 \right)^{\frac{1}{2}}},
\eeas
for some positive constant $C_n$, which is independent of $\tau$. Thus, it holds true that 
$$ \lim_{\tau \to \infty} \abs{\frac{\tilde{d}_n}{(\la_n-\la_\tau)(\tilde{\la}_n-\la_\tau)}} = 0,\quad n \in \N, $$
and consequently that
$$ \lim_{\tau \to \infty} \sum_{n=1}^{N-1} \abs{\frac{\tilde{d}_n}{(\la_n-\la_\tau)(\tilde{\la}_n-\la_\tau)}} = 0,\quad N \in \N. $$
It follows from this and \eqref{dm3} that for all $N \in \N$, 
$$ \limsup_{\tau \to \infty} \left( \sum_{n=1}^\infty \abs{\frac{\tilde{d}_n}{(\la_n-\la_\tau)(\tilde{\la}_n-\la_\tau)}} \right)
\le C \La_N. $$
Therefore
$$ \limsup_{\tau \to \infty} \left( \sum_{n=1}^\infty \abs{\frac{\tilde{d}_n}{(\la_n-\la_\tau)(\tilde{\la}_n-\la_\tau)}} \right)
\le C \inf_{N\in \mathbb{N}}\La_N= C \limsup_{n \to \infty} \abs{\la_n-\tilde{\la}_n},
$$
which, combined with \eqref{dm1}, yields \eqref{dm0}.\\

\noindent {\it Step 2}. For simplicity,  in what follows, the restriction of $\hat{\varrho}$ to $\mathbb{R}^d$ is also denoted by $\hat{\varrho}$.  Having established \eqref{dm0}, we turn now to proving \eqref{se}. First,  from \eqref{dm0},  we get $\hat{\varrho} \in L^\infty(\R^d)$ and satisfies 
\bel{dm5}
\norm{\hat{\varrho}}_{L^\infty(\R^d)} \le C \de,
\ee
where $\de:=\limsup_{n \to \infty} \abs{\la_n-\tilde{\la}_n}$. Further, we have
$$\norm{q-\tilde{q}}_{L^2(\Om)}^2 = \norm{\varrho}_{L^2(\R^d)}^2 = \norm{\hat{\varrho}}_{L^2(\R^d)}^2, $$
from Plancherel's theorem, whence
\bel{dm6}
\norm{q-\tilde{q}}_{L^2(\Om)}^2 = \int_{B_r} \abs{\hat{\varrho}(\xi)}^2\ \dd \xi + \int_{\R^d \setminus B_r} \abs{\hat{\varrho}(\xi)}^2\ \dd \xi,\quad r \in (0,\infty),
\ee
where $B_r := \{ \xi \in \R^d,\ \abs{\xi} < r \}$.

The first term on the right hand side of \eqref{dm6} is easily treated with the aid of \eqref{dm5}, as we have
\bel{dm7} 
\int_{B_r} \abs{\hat{\varrho}(\xi)}^2\ \dd \xi \le C^2 r^d \de^2.
\ee
As for the second term, since
$$
\int_{\R^d \setminus B_r} (1+\abs{\xi}^2)^{-1} \abs{\hat{\varrho}(\xi)}^2\ \dd \xi 
\le r^{-2} \int_{\R^d} \abs{\hat{\varrho}(\xi)}^2\ \dd \xi, $$
we have
\bel{dm8}
\int_{\R^d \setminus B_r} (1+\abs{\xi}^2)^{-1} \abs{\hat{\varrho}(\xi)}^2\ \dd \xi 
\le r^{-2} \norm{q-\tilde{q}}_{L^2(\Om)}^2,
\ee
by Plancherel's theorem. Taking into account that  
$$  \norm{q-\tilde{q}}_{L^2(\Om)} \le c M, $$ 
for some $c=c(d)>0$, it follows from \eqref{dm8}, the inequality
$$ \norm{q-\tilde{q}}_{H^{-1}(\Om)}^2\le \int_{\R^d} (1+\abs{\xi}^2)^{-1} \abs{\hat{\varrho}(\xi)}^2\ \dd \xi$$
and \eqref{dm6}-\eqref{dm7}, that
$$ \norm{q-\tilde{q}}_{H^{-1}(\Om)}^2 \le C \left( r^d \de^2 + r^{-2} \right), $$
where $C$ is a positive constant depending only on $\Om$ and $M$.
Finally, we get \eqref{se} by plugging $r=\de^{-\frac{2}{d+2}}$ in the above estimate.

\subsection{Proof of Theorem \ref{thm2}}
Since $d_n=\tilde{d}_n$ for all $n \ge n_0$, we have
$$
\hat{\varrho}(\xi) = \lim_{\tau \to \infty} \left( \sum_{n=1}^{n_0-1} \frac{d_n-\tilde{d}_n}{\la_n-\la_\tau} + 
\sum_{n=1}^\infty \frac{\tilde{\la}_n-\la_n}{(\la_n-\la_\tau)(\tilde{\la}_n-\la_\tau)} \tilde{d}_n \right) ,\quad \xi \in \R^d,
$$
according to \eqref{s6}. The main task here is to prove that
$$
\lim_{\tau \to \infty} \left( \sum_{n=1}^{n_0-1} \frac{d_n-\tilde{d}_n}{\la_n-\la_\tau} \right) = 0,
$$
thus showing that the identity \eqref{dm0} remains valid in the framework of Theorem \ref{thm2}. Since the sum on the right hand side of the above identity is finite, it is enough to prove for all fixed $n \in \N$, that
\bel{dm9}
\lim_{\tau \to \infty} \frac{d_n-\tilde{d}_n}{\la_n-\la_\tau} = 0.
\ee
For this purpose we aim to estimate $d_n=( g , \psi_n ) ( \psi_n , h )$ 
and $\tilde{d}_n=( g , \tilde{\psi}_n ) ( \tilde{\psi}_n , h )$ 
in terms of $\tau$. We start with $( g , \psi_n )$ and proceed by examining $( g , \psi_n )$ and $( \psi_n , h )$ separately.
To this end we refer to \eqref{go1} and write

\bel{dm9a}
( g , \psi_n  ) = i(\tau+i) (\om_\tau \cdot \nu)\ J_{0,n}(\tau) + J_{1,n}(\tau),
\ee
where
\bel{dm9b}
J_{j,n}(\tau):= \int_{\Ga} e^{i \tau \om_\tau \cdot x} \alpha^{j}(x) e^{-\om_\tau \cdot x}  \psi_n(x) \ \dd x,\quad j=0,1.
\ee
Next, we notice from \eqref{if0.1} that for all $\tau \in \left( \frac{\abs{\xi}}{2} , \infty \right)$,
\bel{dm9c}
\om_\tau = \eta - \frac{\zeta_\tau}{2 \tau},\quad \mbox{where}\; \zeta_\tau := \xi + \frac{\abs{\xi}^2}{1 + \left( 1 - \frac{\abs{\xi}^2}{4 \tau^2} \right)^{\frac{1}{2}}} \frac{\eta}{2 \tau}.
\ee
Keeping in mind that $\abs{\eta}=1$ and $\eta \cdot \xi=0$, we see that
\[ \abs{\zeta_\tau}^2 = 2 \left( 1 + \left( 1 - \frac{\abs{\xi}^2}{4 \tau^2} \right)^{\frac{1}{2}} \right)^{-1} \abs{\xi}^2 ,\]  which yields 
\bel{dm9d} 
\abs{\zeta_\tau} \le \sqrt{2} \abs{\xi}\quad \mbox{and}\quad \abs{\om_\tau} \le 1+ \sqrt{2},\quad \tau \in \left(  \frac{\abs{\xi}}{2} , \infty \right).
\ee
Now, plugging \eqref{dm9c} into \eqref{dm9b}, we obtain for $j=0,1$
\bel{dm9e}
J_{j,n}(\tau) := \int_{\Ga} e^{i \tau \eta \cdot x} a_{j,n,\tau}(x)\ \dd s(x),
\ee
where
\[
a_{j,n,\tau}(x):= \alpha^j(x) e^{-\left( \om_\tau + i \frac{\zeta_\tau}{2} \right) \cdot x}  \psi_n(x).
\]
Further, since $\phi_n \in H^{m+2}(\Om)$ by Lemma \ref{lm-ire} (in fact we have $\phi_n\in H^{m+4}(\Om)$), $\psi_n \in H^m(\Ga)$. Thus we have $a_{j,n,\tau} \in H^m(\Ga)$, $j=0,1$, from \eqref{dm9e}, and the estimate
\bel{dm9e2}
\norm{a_{j,n,\tau}}_{H^m(\Ga)} \le C_\xi \norm{\alpha}_{\cC^m(\Ga)}^j \norm{\psi_n}_{H^m(\Ga)},\quad \tau \in \left( \frac{\abs{\xi}}{2}, \infty \right),
\ee
according to \eqref{dm9d}. Here and in the remaining part of this proof, $C_\xi$ denotes a generic positive constant depending only on $\xi$ and $\Ga$, which may change from line to line.
It follows from \eqref{dm9e2} and Proposition \ref{pr-vdC} that for all $\tau> \frac{\abs{\xi}}{2}$,
$$
\abs{J_{j,n}(\tau)} \le C_\xi \frac{\norm{\alpha}_{\cC^m(\Ga)}^j \norm{\psi_n}_{H^m(\Ga)}}{\tau},\quad  j=0,1,
$$
which, together with \eqref{dm9a} and \eqref{dm9d}, entails that
\bel{dm9f} 
\abs{( g , \psi_n )} \le C_\xi \left(1+\norm{\alpha}_{\cC^m(\Ga)} \right) \norm{\psi_n}_{H^m(\Ga)}.
\ee
Now, with reference to \eqref{go1}-\eqref{go2} and \eqref{if0}, it is clear that we may substitute $h$ for $g$ in \eqref{dm9f} by just arguing in the exact same way as before with $-\xi$ instead of $\xi$. Therefore, using that $(\psi_n, h)=\overline{(h,\psi_n)}$, we find that
$$ \abs{d_n} \le C_\xi^2 \left(1+\norm{\alpha}_{\cC^m(\Ga)} \right)^2 \norm{\psi_n}_{H^m(\Ga)}^2,\quad \tau \in \left( \frac{\abs{\xi}}{2}, \infty \right). $$
The above estimate being still valid upon replacing $(d_n,\psi_n)$ by $(\tilde{d}_n,\tilde{\psi}_n)$, we end up getting that
$$
\abs{\frac{d_n-\tilde{d}_n}{\la_n-\la_\tau}} \le C_\xi^2 \left(1+\norm{\alpha}_{\cC^m(\Ga)} \right)^2 \frac{\norm{\psi_n}_{H^m(\Ga)}^2+\norm{\tilde{\psi}_n}_{H^m(\Ga)}^2}{\left( (\tau^2-\la_n-1)^2 + 4 \tau^2 \right)^{\frac{1}{2}}},
$$
whenever $\tau > \frac{\abs{\xi}}{2}$. This immediately yields \eqref{dm9}. 

As a consequence we have \eqref{dm0} and the rest of the proof is a carbon copy of the one of Theorem \ref{thm1} in Subection \ref{sec-prthm1}.

\subsection{Proof of Corollary \ref{cor3}}
In light of \eqref{i2} we have $-\Delta \phi_n = (\la_n-q) \phi_n$ and $-\Delta \tilde{\phi}_n = (\tilde{\la}_n- \tilde{q}) \tilde{\phi}_n$ in $\Om$ for all $n \in \N$, see e.g. \cite[Eq (1.6)]{CMS}. 
Therefore, putting $\varphi_n:=\phi_n-\tilde{\phi}_n$ and remembering that $\la_n=\tilde{\la}_n$ whenever $n \ge n_0$ and that $q=\tilde{q}$ in $\Om_0$, we get through direct computation that
\bel{cp1} 
-\Delta \varphi_n = W \varphi_n\; \mbox{in}\; \Om_0,\quad n \ge n_0,
\ee
where $W(x): =\la_n-q(x)$, $x \in \Om_0$.
Moreover, we have
\bel{cp2}
 {\varphi_n}_{| \Ga_*}=0,\quad n \ge n_0,
\ee
by assumption, and
\bel{cp3}
\partial_\nu \varphi_n = 0\ \mbox{on}\ \Ga_*,\quad n \ge n_0,
\ee 
since $(\partial_\nu + \al) \phi_n = (\partial_\nu + \al) \tilde{\phi}_n=0$ on $\Ga$, according to \cite[Eq (1.6)]{CMS}. 

Further, using that $\varphi_n \in V=H^1(\Om)$ and that $H^1(\Om_0) \subset L^{p^*}(\Om_0)$ by Sobolev embedding theorem, where $p^*:= 2d \slash (d-2)$ is greater that its conjugate number $p:=2d \slash (d+2)$, we see that $\varphi_n \in L^p(\Om_0)$ for all $n \ge n_0$.
From this, $W \in L^{d \slash2}(\Om_0)$  and \eqref{cp1}--\eqref{cp3}, we get by unique continuation (see, e.g. \cite[Lemma 3.1]{C3}), that
$$ \varphi_n = 0\ \mbox{in}\ \Om_0,\quad n \ge n_0. $$
Therefore, we have ${\varphi_n}_{| \Ga}=0$ and hence
$$ \psi_n=\tilde{\psi}_n\ \mbox{in}\ \Ga,\quad n \ge n_0. $$
Finally, the desired result follows readily from this and from Corollary \ref{cor2}.

\section{A Van-der-Corput type inequality}
\label{sec-VdC}

In this section we establish a Van-der-Corput type inequality in a fairly general setting. This estimate is a key to the proof of Theorem \ref{thm2}.

\subsection{Settings, notations and statement of the result}
\label{sec-VdC-settings}

Fix $s\in (\frac{d-1}{2},\infty)$, $d \ge 3$,  let $m=\lceil s\rceil +1$, and  assume that $\Omega$ is a $\cC^{m+2}$ open bounded domain of $\R^d$. Its boundary $\Ga$ is described by a local chart of bounded open subsets $\cO_j \subset \R^d$,
$\om_j \subset \R^{d-1}$, and functions $\vartheta_j \in \cC^m(\overline{\om_j})$, $j=1,\ldots, J$, where $J \in \N$ is fixed, such that we have \eqref{loc-chart}. 
Moreover, we recall that a generic point $x \in \R^d $ is written $x=(\xp,x_d)$ with $\xp=(x_1,\ldots,x_{d-1})$.

\subsubsection{Phase and stationary points}

For $\tea \in \mathbb{S}^{d-1}$ fixed, we introduce
\bel{a1}
\varphi_{j}(\xp):= \teap \cdot \xp + \tea_d \vartheta_j(\xp),\ \xp \in \om_j,\ j=1,\ldots,J.
\ee
Since $\nap \varphi_{j}(\xp)= \teap + \tea_d \nap \vartheta_j(\xp)$ for all $\xp \in \om_j$,
we see that $\xp \in \om_j$ is a stationary point of $\varphi_j$ if and only if
$\teap =- \tea_d \nap \vartheta_j(\xp)$. In light of \eqref{b1}, this may be equivalently rewritten as
$$
\tea 
=-\tea_d \sqrt{1+\abs{\nap \vartheta_j(\xp)}^2} \nu_j(\xp),
$$
showing that the stationary points of $\varphi_j$ are 
those $\xp \in \omega_j$ where the normal vector $\nu_j(\xp)$ is parallel to $\tea$.

\subsubsection{Statement of the result}

The main result of this section is the following Van-der-Corput type inequality. Prior to stating it, we remind the reader that the notation $\lceil s \rceil$ stands for the smallest natural number greater or equal to $s\in (0,\infty)$.

\begin{proposition}
\label{pr-vdC}
Suppose in addition that $\nap \varphi_j$ vanishes within $\om_j$ only at $\xp_{j,k}$ for $k=1,\ldots,N_j$ and some $N_j \in \N$, and that the Hessian matrix
$$ {\mathbb D^\prime}^2 \varphi_j(\xp_{j,k}) = \left( \pd_{\xp_p,\xp_q}^2 \varphi_j(\xp_{j,k})\right)_{1 \le p,q \le d-1} $$
is non-singular. Then, there exists a constant $C>0$, depending only on $\Om$, such that for all $\phi \in H^s(\Ga)$ and all $\tau \in (1,\infty)$, we have
\bel{ivdc} 
\abs{\int_{\Ga} e^{i \tau \tea \cdot x} \phi(x) \dd x} \le C \frac{\norm{\phi}_{H^s(\Ga)}}{\tau}.
\ee
\end{proposition}

\subsection{Proof of Proposition \ref{pr-vdC}}
The statement of Proposition \ref{pr-vdC} being obviously true when $\phi$ vanishes everywhere in $\Gamma$, we will assume in the following that $\phi$ is not identically zero in $\Gamma$.
The proof of Proposition \ref{pr-vdC} is quite lengthy so we break it down in several parts.

\subsubsection{Boundary integral decomposition}
Let $\{ \chi_j,\ j=0,1,\ldots,J \}$ be a partition of the unity subordinate to $\{ \cO_j,\ j=1,\ldots,J \}$. Otherwise stated, we pick $J+1$ functions $\chi_j \in C^\infty(\R^d,[0,1])$, $j=0,1,\ldots,J$, such that
$$ \supp(\chi_0) \subset \R^d \setminus \Ga\ \mbox{and}\ \chi_j \in C_0^\infty(\cO_j)\ \mbox{for}\ j=1,\ldots,J,$$
and
$$\sum_{j=0}^J \chi_j(x)=1,\quad x \in \R^d.$$
Then, we have 

$$ \int_{\Ga} v(x) \dd s(x) =  \sum_{j=1}^J \int_{\om_j} \left( \chi_j v \right)(\xp,\vartheta_j(\xp)) \sqrt{1+\abs{\nap \vartheta_j(\xp)}^2} \dd \xp,$$
whenever $v \in L^1(\Ga)$. As a consequence, for all
$\tea \in \mathbb{S}^{d-1}$ and all $\tau \in (0,\infty)$, we find upon substituting $e^{i \tau \tea \cdot x} \phi(x)$ by $v(x)$ in the above equality, that
\bel{a0}  
\int_{\Ga} e^{i \tau \tea \cdot x} \phi(x) \dd s(x) =  
\sum_{j=1}^J \int_{\om_j} e^{i \tau \varphi_j(\xp)} w_j(\xp) \dd \xp.
\ee
Here, $\varphi_j$, $j=1,\ldots,J$, is the same as in \eqref{a1} and
\bel{a2}
w_j(\xp):=\left( \chi_j \phi \right)(\xp,\vartheta_j(\xp)) \sqrt{1+\abs{\nap \vartheta_j(\xp)}^2},\quad \xp \in \om_j.
\ee

Further, we choose $\delta \in (0,\infty)$ so small that the sets 
$$U(\xp_{j,k},\delta):=\{ \xp \in \om_j,\ \abs{\xp-\xp_{j,k}} < \delta \},\quad k=1,\ldots,N_j, $$
are pairwise disjoint, i.e., such that
$$ U(\xp_{j,k},\delta) \cap U(\xp_{j,\ell},\delta) = \emptyset,\quad 1 \le k \neq \ell \le N_j. $$
Then we set
\bel{d1}
U_j=U_j(\delta):= \bigcup_{k=1}^{N_j} U(\xp_{j,k},\delta)\quad \mbox{and}\quad
\tilde{U}_j=\tilde{U}_j(\delta):=\om_j \setminus U_j,
\ee 
in such a way that we have
\bel{d2}
\int_{\om_j} e^{i \tau \varphi_j(\xp)} w_j(\xp) \dd \xp = I_j(\tau) + \tilde{I}_j(\tau),
\ee
where
\bel{d3}
I_j(\tau):=\int_{U_j} e^{i \tau \varphi_j(\xp)} w_j(\xp) \dd \xp\quad \mbox{and}\quad
\tilde{I}_j(\tau) := \int_{\tilde{U}_j} e^{i \tau \varphi_j(\xp)} w_j(\xp) \dd \xp.
\ee
We shall study each of the two integrals $I_j(\tau)$ and $\tilde{I}_j(\tau)$ separately. 
We start with $\tilde{I}_j(\tau)$.

\subsubsection{Analysis of $\tilde{I}_j(\tau)$}
\label{sec-atI}
The analysis of the integral $\tilde{I}_j(\tau)$ is quite straightforward as it is based on a Riemann-Lebesgue's type argument.
Indeed, since the function $\varphi_j$ is continuously differentiable in $\tilde{U}_j$ and has no stationary point in the compact set $\overline{\tilde{U}_j}$, there exists $\kappa \in (0,\infty)$ such that
\bel{e0} 
\abs{\nap \varphi_j(\xp)} \ge \kappa >0,\ \xp \in \tilde{U}_j.
\ee
Thus, we have
$$ e^{i \tau \varphi_j(\xp)} = \frac{1}{i \tau} \frac{\nap \left( e^{i \tau \varphi_j(\xp)} \right) \cdot \nap \varphi_j(\xp)}{\abs{\nap \varphi_j(\xp)}^2},\ \xp \in \tilde{U}_j, $$
and consequently
\beas
\tilde{I}_j(\tau) & = & \frac{1}{i\tau} \left( -\int_{\tilde{U}_j} e^{i \tau \varphi_j(\xp)} \nap \cdot \left( w_j(\xp) \frac{\nap \varphi_j(\xp)}{\abs{\nap \varphi_j(\xp)}^2} \right) \dd \xp \right. \\
& & + \left. \int_{\pd \tilde{U}_j} e^{i \tau \varphi_j(\ga^\prime)} w_j(\ga^\prime) \frac{\nap \varphi_j(\ga^\prime) \cdot \tilde{n}_j(\ga^\prime)}{\abs{\nap \varphi_j(\ga^\prime)}^2}  \dd \ga^\prime \right),
\eeas
from the divergence formula, where $\tilde{n}_j$ stands for the outgoing normal vector to 
$\pd \tilde{U}_j$. 
Therefore, keeping in mind that the function $\varphi_j$ is real valued according to \eqref{a1}, we deduce from the above identity that
\bel{e1}
\abs{\tilde{I}_j(\tau)} \le   \frac{1}{\tau} \left( \norm{\nap  \left(  \frac{w_j \nap \varphi_j}{\abs{\nap \varphi_j}^2}\right)}_{L^1(\tilde{U}_j)} + \norm{\frac{w_j \pd_{\tilde{n}_j} \varphi_j}{\abs{\nap \varphi_j}^2}}_{L^1(\pd \tilde{U}_j)}\right),
\ee
where $\pd_{\tilde{n}_j} \varphi_j:=\nap \varphi_j \cdot \tilde{n}_j$ denotes the normal derivative of $\varphi_j$ with respect to $\tilde{n}_j$. Furthermore, since
$$ \nap  \left(  \frac{w_j \nap \varphi_j}{\abs{\nap \varphi_j}^2} \right) \\
= \frac{\nap w_j \cdot \nap \varphi_j + w_j \Delta^\prime \varphi_j}{\abs{\nap \varphi_j}^2} - 2 \frac{w_j {\mathbb D^\prime}^2 \varphi_{j} \nap \varphi_j \cdot \nap \varphi_j}{\abs{\nap \varphi_j}^4},
$$
it follows from \eqref{e0}-\eqref{e1} and the Cauchy-Schwarz inequality that
\beas
\abs{\tilde{I}_j(\tau)} & \le & \frac{\kappa^{-2}}{\tau} \left( 2 \left( \norm{\varphi_j}_{H^2(\tilde{U}_j)} + \kappa^{-2} \norm{{\mathbb D^\prime}^2 \varphi_{j} \nap \varphi_j \cdot \nap \varphi_j}_{L^2(\tilde{U}_j)}  \right) \norm{w_j}_{H^1(\tilde{U}_j)} \right. \nonumber \\
& & \left. + 
\norm{\pd_{\tilde{n}_j} \varphi_j}_{L^2(\pd \tilde{U}_j)} \norm{w_j}_{L^2(\pd \tilde{U}_j)} \right) \nonumber \\
& \le & \frac{\kappa^{-2}}{\tau}  \left( (2+c_0 c_1) \norm{\varphi_j}_{H^2(\tilde{U}_j)} + 2 \kappa^{-2} \norm{{\mathbb D^\prime}^2 \varphi_{j} \nap \varphi_j \cdot \nap \varphi_j}_{L^2(\tilde{U}_j)} \right)  \norm{w_j}_{H^1(\tilde{U}_j)}, 
\eeas
where $c_0$ (resp., $c_1$) stands for the norm of the linear bounded trace operator 
$u \in H^1(\tilde{U}_j) \mapsto u_{| \pd \tilde{U}_j} \in L^2(\pd \tilde{U}_j)$ (resp., $u \in H^2(\tilde{U}_j) \mapsto \pd_{\tilde{n}_j} u \in L^2(\pd \tilde{U}_j))$.  Now, in light of \eqref{a2}, we have

\begin{align*}
\nap w_j(\xp) 
& =  \Big[ \left( \phi \nap \chi_j + \chi_j \nap \phi \right) \circ \Psi_j(\xp).  %\right] \sqrt{1+\abs{\nap \psi_j(\xp)}^2}
 \\
 &\hskip 1.8cm +   
 \left( \phi \pd_{\xp_d} \chi_j + \chi_j \pd_{\xp_d} \phi \right) \circ \Psi_j(\xp) \nap \vartheta_j(\xp) \Big] \sqrt{1+\abs{\nap \vartheta_j(\xp)}^2}
 \\
&\hskip3.5cm  + \frac{(\chi_j \phi) \circ \Psi_j(\xp)}{\sqrt{1+\abs{\nap \vartheta_j(\xp)}^2}}  {\mathbb D^\prime}^2 \vartheta_j(\xp) \nap \vartheta_j(\xp),\quad \xp \in \om_j,
\end{align*}
where we used the notation $\Psi_j(\xp):=(\xp,\vartheta_j(\xp))$. Whence
\bel{e2}
\norm{w_j}_{H^1(\tilde{U}_j)} \le 2\left( 1 + \norm{\vartheta_j}_{\cC^2(\overline{\om_j})} \right)^2 \norm{\chi_j}_{\cC^1(\overline{\cO_j})} \norm{\phi}_{H^1(\Ga_j)}.
\ee
Similarly, since $\norm{\varphi_j}_{H^2(\tilde{U}_j)}$ and $\norm{{\mathbb D^\prime}^2 \varphi_{j} \nap \varphi_j \cdot \nap \varphi_j}_{L^2(\tilde{U}_j)}$ can both be estimated in terms of $\norm{\psi_j}_{\cC^2(\overline{\om_j})}$ and the Lebesgue measure $\abs{\om_j}$ of $\om_j$, 
we end up getting that
\begin{equation}
\label{e3}
\abs{\tilde{I}_j(\tau)} \le \frac{\tilde{C}_j}{\tau}  \norm{\phi}_{H^1(\Ga_j)}, 
\end{equation}
for some positive constant $\tilde{C}_j$, depending only on $\kappa$, $\abs{\om_j}$, $\norm{\chi_j}_{\cC^1(\overline{\cO_j})}$ and $\norm{\vartheta_j}_{\cC^2(\overline{\om_j})}$. 

\begin{remark}
\label{rmk}
Using that $\vartheta_j \in \cC^{2+m}(\overline{\omega_j})$ with $m = \lceil s \rceil +1$, and that $\chi_j \in \cC^\infty(\cO_j)$, we deduce from \eqref{a2} upon arguing in the same fashion as in the derivation of \eqref{e2}, that $w_j \in H^s(\om_j)$ satisfies the estimate
\bel{k1}  
\norm{w_j}_{H^s(\om_j)} \le c_j \norm{\phi}_{H^s(\Ga_j)},
\ee
for some positive constant $c_j$ depending only on $\vartheta_j$ and $\chi_j$, i.e., on the boundary $\Gamma$.
\end{remark}

This being said, we turn now to studying $I_j(\tau)$.

\subsubsection{Analysis of $I_j(\tau)$}

The main purpose of this section is to prove existence of a positive constant $C_j$, which is independent of $\tau$ and $\phi$, such that the estimate
\bel{g5}
\abs{I_j(\tau)} \le C_j \tau^{-\frac{d-1}{2}} \norm{\phi}_{H^s(\Ga_j)} %,\ \tau \in (0,\infty), 
\ee
holds uniformly in $\tau \in (0,\infty)$.
In light of \eqref{d1} and \eqref{d3}, this will be achieved upon examining each
\bel{f1} 
I_{j,k}(\tau):=\int_{U_j(\xp_{j,k},\delta)} e^{i \tau \varphi_j(\xp)} w_j(\xp) \dd \xp,\quad k=1,\ldots,N_j,
\ee
separately. 

Moreover, we shall assume without loss of generality that $w_j$ is compactly supported in $U_j(\xp_{j,k},\delta)$. The reason is as follows. If we pick $\zeta \in C_0^\infty(U_j(\xp_{j,k},\delta),[0,1])$ satisfying $\zeta(\xp)=1 $ for all $\xp \in U_j(\xp_{j,k},\delta \slash 2)$, write
$$ I_{j,k}(\tau)=\int_{U_j(\xp_{j,k},\delta)} e^{i \tau \varphi_j(\xp)} \zeta(\xp) w_j(\xp) \dd \xp + \int_{\hat{U}_{j,k}(\delta)} e^{i \tau \varphi_j(\xp)} (1-\zeta(\xp)) w_j(\xp) \dd \xp, $$
where
$$ \hat{U}_{j,k}(\delta) := U_j(\xp_{j,k},\delta) \setminus U_j(\xp_{j,k},\delta \slash 2) = \left\{ \xp \in \om_j,\ \delta \slash 2 \le \abs{\xp-\xp_{j,k}} < \delta \right\},
$$
and take into account that 
$$\abs{\nap \varphi_j(\xp)} \ge c_{j,k} >0,\ \xp \in  \hat{U}_{j,k}(\delta), $$
we get upon arguing in the same way as in Section \ref{sec-atI}, that the estimate \eqref{e3} remains valid by substituting 
$\int_{U_j(\xp_{j,k},\delta)} e^{i \tau \varphi_j(\xp)} (1-\zeta(\xp)) w_j(\xp) \dd \xp$ for $\tilde{I}_j(\tau)$.
Therefore, upon possibly replacing $w_j$ by $\zeta w_j$ in \eqref{f1}, we may assume in the remaining part of this section that $w_j$ is compactly supported in $U_j(\xp_{j,k},\delta)$.\\

 \noindent (a)  {\it  Quadratic phase.}
By Morse lemma, there exist $\eps \in (0,1)$,  an open subset $V_{j,k}(\eps) \subset \R^{d-1}$, and a
$\cC^m$-diffeomorphism $\ga : B(\xp_{j,k},\eps) \to V_{j,k}(\eps)$, where $B(\xp_{j,k},\eps):= \{ \xp \in \R^{d-1},\ \abs{\xp-\xp_{j,k}}<\eps \}$, such that
\bel{m1} 
\ga(\xp_{j,k})=0
\ee
and
\bel{m2}
\varphi_j(\xp)=\varphi_j(\xp_{j,k})  - \sum_{i=1}^{{\mathrm{ind}}_{j,k}} \ga_i(\xp)^2 + \sum_{i={\mathrm{ind}}_{j,k}+1}^{d-1} \ga_i(\xp)^2.
\ee
Here, ${\mathrm{ind}}_{j,k}$ is the index (i.e. the number of negative eigenvalues) of the Hessian matrix ${\mathbb D^\prime}^2 \varphi_j(\xp_{j,k})$ and $\ga_i$ stands for $i$-th component of the function $\ga$, $i=1,\ldots,d-1$. Moreover, the first sum on the right hand side of \eqref{m2} is taken equal to zero when ${\mathrm{ind}}_{j,k}=0$. 
We point out that \eqref{m1}-\eqref{m2} can be obtained for instance by combining \cite[Lemma 6.3.1]{Jo} and \cite[Chap. 1, Proof of Lemma 2.2]{Mil}.

Let us denote by $\eta(\yp)=(\eta_i(\yp))_{1 \le i \le d-1}$, $\yp \in V_{j,k}(\eps)$, the inverse function to $\ga$ and let ${\mathbb D}^\prime \eta(\yp):=\left( \pd_{\yp_q} \eta_p(\yp) \right)_{1 \le p,q \le d-1}$ be the Jacobian matrix of $\eta$ at $\yp \in \R^{d-1}$. Then, upon possibly shortening $\delta>0$ so that $\delta \in (0,\eps]$, we get from \eqref{m1}-\eqref{m2} by performing the change of variable $\yp=\ga(\xp)$ in \eqref{f1}, that
\bea
I_{j,k}(\tau) & = & 
\int_{\ga(U_j(\xp_{j,k},\delta))} e^{i \tau \varphi_j(\eta(\yp))} w_j(\eta(\yp)) \abs{\det({\mathbb D}^\prime \eta(\yp))} \dd \yp \nonumber \\
& = & e^{i \tau \varphi_j(\xp_{j,k})} \int_{\mathbb{R}^{d-1}} e^{i B_\tau \yp \cdot \yp} h(\yp) \dd \yp. \label{m4}
\eea
Here, we have set
\bel{m5}
B_\tau:=\tau D,\quad \tau \in (0,\infty),
\ee 
where $D$ is the $(d-1)$-dimensional diagonal matrix whose ${\mathrm{ind}}_{j,k}$ first entries are $-1$ and the other ones are $1$, and
\bel{m6}
h(\yp)=h_{j,k}(\yp):= 1_{\ga(U_j(\xp_{j,k},\delta))}(\yp) w_j(\eta(\yp)) \abs{\det({\mathbb D}^\prime \eta(\yp))},\quad \yp \in \R^{d-1}, 
\ee
the notation $1_{\cO}$ being used for the characteristic function of $\cO \subset \R^{d-1}$. Keeping in mind that $w_j \in H^s(\om_j)$ according to \eqref{a2} since $\chi_j \in C^\infty(\R^d)$, that $\vartheta_j \in C^{2+m}(\overline{\om_j})$ with $m=\lceil s \rceil +1$ and $\phi \in H^s(\Ga)$, and that $\eta \in \cC^m(V_{j,k}(\eps),B(\xp_{j,k},\eps))$, 
it follows readily from \eqref{m6} that
$h \in H^s(\R^{d-1})$ satisfies
\bel{k2} 
\norm{h}_{H^s(\R^{d-1})} \le \tilde{c}_j \norm{w_j}_{H^s(\om_j)},
\ee
for some positive constant $\tilde{c}_j$, depending only on $\Ga$.

Next, following the idea of the stationary phase method (see, e.g., \cite[Chapter 10, Section 9]{Zu}), we shall rewrite the integral on the right hand side of \eqref{m4} in terms of the Fourier transform of $\yp \mapsto e^{i \yp \cdot B_\tau \yp}$, $\tau \in (0,\infty)$, which can be easily estimated in terms of $\tau$. \\

\noindent (b) {\it  Fourier transform.}
Since the matrix $B_\tau$ is non-singular, we have (see, e.g.  \cite[Eq. (3.6)]{Zu}),
\bel{z0} 
\cF \left( e^{i \yp \cdot B_\tau \yp} \right)(\xp)= \frac{\pi^{\frac{d-1}{2}}}{\abs{\det B_\tau}^{\frac{1}{2}}} e^{i \frac{\pi}{4} \sgn(B_\tau)} e^{-\frac{i}{4} \xp \cdot B_\tau^{-1} \xp},\ \xp \in \R^{d-1}.
\ee
Here, $\cF$ denotes the Fourier transform in $\cS^\prime(\R^{d-1})$, the space of tempered distributions in $\R^{d-1}$, and $\sgn(B_\tau):=d-1-2{\mathrm{ind}}_{j,k}$ is the signature of $B_\tau$. 

But, as $h$ does not lie in the Schwartz class $\cS(\R^{d-1})$, it cannot be directly replaced by its inverse Fourier transform, while the right hand side of 
\eqref{z0} is substituted for $e^{i \yp \cdot B_\tau \yp}$, in the last integral of \eqref{m4}. For that reason, we introduce a mollifying sequence $(h_n)$ of $h$.\\
\\
\noindent  (c) {\it Mollifier.}
Let $\rho$ be a mollifier. We mean by this that $\rho$ is a nonnegative $\cC^\infty(\R^{d-1})$ function supported in the closure of $B(0,1)=\{ \xp \in \R^{d-1},\ \abs{\xp} < 1 \}$, that is normalized in $L^2(\R^{d-1})$. For instance, 
$$\rho(\xp)= \left\{ \begin{array}{cl} c e^{-\frac{1}{1-\abs{\xp}^2}} & \mbox{if}\ \xp \in B(0,1), \\ 0 & \mbox{otherwise}, \end{array} \right. $$ 
where $c=\left( \int_{\abs{\xp}<1} e^{-\frac{1}{1-\abs{\xp}^2}} \dd \xp \right)^{-1}$ is a mollifier. For all $n \in \N$, set
$$\rho_n(x):=n^{d-1} \rho(nx),\ x \in \R^{d-1}. $$
Since $h \in L_{\rm loc}^1(\R^{d-1})$ and $\rho_n \in C_0^\infty(\R^{d-1})$, we have $h_n=h \ast \rho_n \in C^{\infty}(\R^{d-1})$. Further, the functions $h$ and $\rho_n$ being supported in $U:=\overline{\eta^{-1}(U_j(\xp_{j,k},\delta))}$ and $ \overline{B(0,1 \slash n)}$, respectively, it holds true that  $\supp(h_n) \subset \{ \yp \in \R^{d-1},\ \dist (\yp,U) \le 1 \slash n \}$. Therefore, we see that $h_n \in C_0^{\infty}(\R^{d-1})$ with
\bel{g0}
\supp(h_n) \subset U_1:= \{ \yp \in \R^{d-1},\ \dist(\yp,U) \le 1 \},\ n \in \N,
\ee
and we recall for further use that
\bel{g1}
\lim_{n \to \infty} \norm{h-h_n}_{L^1(\R^{d-1})}=0.
\ee
Moreover, putting 
\bel{g3}
I_{j,k,n}(\tau):= e^{i \tau \varphi_j(\xp_{j,k})} \int_{\R^{d-1}} e^{i \yp \cdot B_\tau \yp} h_n(\yp)  \dd \yp,\quad n \in \N,\; \tau \in (0,\infty),
\ee
we have the following technical result, whose proof is postponed to Subsection \ref{sec-pr-lm2}.
\begin{lemma}
\label{lm2}
For all $\tau \in (0,\infty)$ and all $n \in \N$, it holds true that
$$
\abs{I_{j,k,n}(\tau)} \le (2 \pi)^{\frac{d-1}{2}} C_s \tau^{-\frac{d-1}{2}}  \norm{\rho}_{L^1(\R^{d-1})} \norm{h}_{H^{s}(\R^{d-1})},
$$
where $C_s$ is the constant defined in Lemma \ref{lm1}.
\end{lemma}
Armed with Lemma \ref{lm2}, we are in position to estimate $I_{j,k}(\tau)$ in terms of $\tau$.\\

\noindent (d) {\it  Estimation of $I_{j,k}(\tau)$.} Since $\varphi_j(\xp_{j,k}) \in \R$ and $\yp \cdot B_\tau \yp \in \R$ for all $\yp \in \R^{d-1}$, we have
\beas
\abs{I_{j,k}(\tau) - I_{j,k,n}(\tau)} & = & \abs{\int_{\R^{d-1}} e^{i \yp \cdot B_\tau \yp} (h(\yp) )-h_n(\yp)) \dd \yp} \\
& \le & \norm{h-h_n}_{L^1(\R^{d-1})},\quad \tau \in (0,\infty),
\eeas
by \eqref{m4} and \eqref{g3}, and consequently 
$$\lim_{n \to \infty} \abs{I_{j,k}(\tau) - I_{j,k,n}(\tau)}=0,\quad \tau \in (0,\infty), $$
from \eqref{g1}.
Thus, for all $\tau \in (0,\infty)$, there exists $n_1=n_1(\tau,h) \in \N$, such that
$$
\abs{I_{j,k}(\tau) - I_{j,k,n}(\tau)} \le \tau^{-\frac{d-1}{2}} \norm{h}_{H^s(\R^{d-1})},\quad n \ge n_1.
$$
Here we used that $h$ is non identically equal to zero in $\R^{d-1}$, as can be seen from \eqref{a2}, \eqref{m6} and the assumption that $\Phi$ does not vanish everywhere in $\Om$. Now, applying Lemma \ref{lm2} with $n=n_1$, we deduce from the above 
estimate that
$$
\abs{I_{j,k}(\tau)} \le C_{j,k} \tau^{-\frac{d-1}{2}} \norm{h}_{H^{s}(\R^{d-1})},\quad \tau \in (0,\infty), 
$$
where $C_{j,k}$ is a generic positive constant independent of $\tau$, which may change from line to line.
Putting this together with \eqref{k1} and \eqref{k2}, we find that
$$
\abs{I_{j,k}(\tau)} \le C_{j,k} \tau^{-\frac{d-1}{2}} \norm{\phi}_{H^s(\Ga_j)},\quad\tau \in (0,\infty),
$$
which, in turn, yields \eqref{g5}.

\subsubsection{End of the proof}
In light of \eqref{a0}, \eqref{d2}, \eqref{e3} and \eqref{g5}, we have
$$
\abs{\int_{\Ga} e^{i \tau \om \cdot x} \phi(x) \dd \ga} \le C \left( \frac{\norm{\phi}_{H^1(\Ga)}}{\tau} + \frac{\norm{\phi}_{H^s(\Ga)}}{\tau^{\frac{d-1}{2}}} \right),\quad \tau \in (0,\infty),
$$
for some positive constant $C$  independent of $\tau$ and $\phi$. 
Evidently, \eqref{ivdc} follows from this as we have $(d-1) \slash 2 \ge 1$ and
$s>(d-1) \slash 2$.

\subsubsection{Proof of Lemma \ref{lm2}}
\label{sec-pr-lm2}
Fix $n \in \N$ and $\tau \in (0,\infty)$. With reference to \eqref{g0}, we have $h_n \in C_0^\infty(\R^{d-1}) \subset \cS(\R^{d-1})$, whence
\beas
I_{j,k,n}(\tau) & = & e^{i \tau \varphi_j(\xp_{j,k})} \int_{\R^{d-1}} e^{i \yp \cdot B_\tau \yp} h_n(\yp)  \dd \yp \\
& = & e^{i \tau \varphi_j(\xp_{j,k})} \langle e^{i \yp \cdot B_\tau \yp} , h_n \rangle_{\cS^\prime(\R^{d-1}),\cS(\R^{d-1})},
\eeas
where $\langle \cdot , \cdot \rangle_{\cS^\prime(\R^{d-1}),\cS(\R^{d-1})}$ denotes the duality pairing between $\cS(\R^{d-1})$ and its dual space $\cS^\prime(\R^{d-1})$. Therefore, it holds true that
\beas
I_{j,k,n}(\tau) & = & e^{i \tau \varphi_j(\xp_{j,k})} \langle \cF(e^{i \yp \cdot B_\tau \yp}) , \cF^{-1}(h_n) \rangle_{\cS^\prime(\R^{d-1}),\cS(\R^{d-1})} \\
& = & e^{i \tau \varphi_j(\xp_{j,k})} \int_{\R^{d-1}} \cF(e^{i \yp \cdot B_\tau \yp})(\xp) \cF(h_n)(-\xp)  \dd \xp,
\eeas
where we used  in the last line that $\cF^{-1}(h_n)(\xp)=\cF(h_n)(-\xp)$ for all $\xp \in \R^{d-1}$. From this and \eqref{z0}, it then follows that
\begin{equation}
\label{g2}
I_{j,k,n}(\tau) =  e^{i \tau \varphi_j(\xp_{j,k})}  
\frac{\pi^{\frac{d-1}{2}}}{\abs{\det B_\tau}^{\frac{1}{2}}} e^{i \frac{\pi}{4} \sgn(B_\tau)} \int_{\R^{d-1}} e^{-\frac{i}{4} \xp \cdot B_\tau^{-1} \xp} \cF(h_n)(-\xp)  \dd \xp.
\end{equation}
Next, since $\det B_\tau$ scales like $\tau^{d-1}$, i.e.,
\bel{g2b}
\det B_\tau = \det (D) \tau^{d-1},
\ee
according to \eqref{m5}, it suffices to show that the integral on the right hand side of \eqref{g2} can be bounded 
by a constant  independent of $\tau$. To this end, we recall from $h_n=\rho_n \ast h$ that $\cF(h_n)=\cF(\rho_n) \cF(h)$ and write 
\begin{align*}
&\abs{\int_{\R^{d-1}} e^{-\frac{i}{4} \xp \cdot B_\tau^{-1} \xp} \cF(h_n)(-\xp)  \dd \xp} 
\\
&\hskip 2cm =  \abs{\int_{\R^{d-1}} e^{-\frac{i}{4} \xp \cdot B_\tau^{-1} \xp} 
\cF(\rho_n)(-\xp) \cF(h)(-\xp) \dd \xp} 
\\
&\hskip2cm \le  \norm{\cF(\rho_n)}_{L^\infty(\R^{d-1})} \norm{\cF(h)}_{L^1(\R^{d-1})}.
\end{align*}
Thus, applying the estimate \eqref{f0} in Lemma \ref{lm1}, we find that
\beas
\abs{\int_{\R^{d-1}} e^{-\frac{i}{4} \xp \cdot B_\tau^{-1} \xp} \cF(h_n)(-\xp)  \dd \xp}
& \le & C_s \norm{\cF(\rho_n)}_{L^\infty(\R^{d-1})} \norm{h}_{H^{s}(\R^{d-1})},
\eeas
where the constant $C_s$ depends only on $s$. Moreover, since $\cF(\rho_n)(\xp)=\cF(\rho)(\xp \slash n)$, we have
$$\norm{\cF(\rho_n)}_{L^\infty(\R^{d-1})}=\norm{\cF(\rho)}_{L^\infty(\R^{d-1})} \le \norm{\rho}_{L^1(\R^{d-1})}, $$
whence
$$
\abs{\int_{\R^{d-1}} e^{-\frac{i}{4} \xp \cdot B_\tau^{-1} \xp} \cF(h_n)(-\xp)  \dd \xp} 
\le C_s \norm{\rho}_{L^1(\R^{d-1})} \norm{h}_{H^{s}(\R^{d-1})}.
$$
Combining this with \eqref{g2}-\eqref{g2b}, we end up getting the expected result.

\section*{Acknowledgements}
Two of the authors of this article (AM \& \'ES) thank Hiroshi Izosaki (University of Tsukuba, Japan) for interesting discussions on this work.

\appendix

\section{$L^1$-estimate of the Fourier transform of $H^s(\R^{d-1})$-functions} 
\begin{lemma}
\label{lm1}
For all $h \in H^s(\R^{d-1})$ with $s >\frac{d-1}{2}$, we have
$\cF(h) \in L^1(\R^{d-1})$ and the estimate
\bel{f0} 
\norm{\cF(h)}_{L^1(\R^{d-1})} \le \sqrt{C_s} \norm{h}_{H^{s}(\R^{d-1})},
\ee
where $C_s:=\int_{\R^{d-1}} (1+\abs{\xp}^2)^{-s} \dd \xp <\infty$.
\end{lemma}
\begin{proof}

Let $h\in H^s(\R^{d-1})$.
Applying the Cauchy-Schwarz inequality, we get 
\beas
\int_{\R^{d-1}} \abs{\cF(h)(\xp)} \dd \xp 
& = & \int_{\R^{d-1}} (1+\abs{\xp}^2)^{-\frac{s}{2}} (1+\abs{\xp}^2)^{\frac{s}{2}} \abs{\cF(h)(\xp)} \dd \xp \\
& \le & \sqrt{C_s} \norm{(1+\abs{\xp}^2)^{\frac{s}{2}} \cF(h)}_{L^2(\R^{d-1})}.
\eeas
The expected inequality then follows. 
\end{proof}

\section{Elliptic regularity}
\subsection{An elliptic regularity result} 
Consider the BVP:
\bel{r0}
\left\{ \begin{array}{ll} -\Delta u + q u = f & \mbox{in}\ \Om, \\ \pd_\nu u + \al u = 0 & \mbox{on}\ \Gamma. \end{array} \right. 
\ee
\begin{lemma}
\label{lm-er}
Fix $k \in \N_0 := \N \cup \{ 0 \}$ and $r \in \left( \frac{1}{2}, 1\right)$. Let $\Gamma$ be $\cC^{k+1,1}$, let $q \in W^{k,\infty}(\Om)$ and let $\al \in \cC^{k,r}(\Ga)$. Suppose moreover that $q$ and $\al$ satisfy either of the two following conditions 
\bel{c1b}
q \ge 0\; \mbox{in}\; \Om\quad \mbox{and}\quad \al \ge c >0\; \mbox{on}\ \Ga
\ee
or
\bel{c2b}
q \ge c > 0\ \mbox{in}\; \Om\quad \mbox{and}\quad \al \ge 0\; \mbox{on}\; \Ga.
\ee
Then, for all $f \in H^k(\Om)$, the BVP \eqref{r0} admits a unique solution $u \in H^{k+2}(\Om)$.
\end{lemma}
\begin{proof}
The proof is by induction on $k$. First, let $k=0$ and $f \in L^2(\Om)$. Applying either \cite[Theorem 2.4.2.6]{G} or \cite[Theorem 2.4.2.7]{G}, depending on whether $q$ and $\al$ satisfy the condition \eqref{c1b} or \eqref{c2b}, we get a unique solution $u \in H^2(\Om)$ to \eqref{r0}.
Assume that the lemma holds up to $k-1$ with $k\in \mathbb{N}$. If $f\in H^k(\Omega)$, we obtain 
$u \in H^{k+1}(\Om)$, and hence that $u_{| \Ga} \in H^{k+\frac{1}{2}}(\Ga)$. It follows from this and $\al \in \cC^{k,r}(\Ga)$, that $\al u \in H^{k+\frac{1}{2}}(\Ga)$, according to \cite[Theorem 2.3.9]{Ag}. Thus, keeping in mind that $q u \in H^k(\Om)$ since $q \in W^{k,\infty}(\Omega)$, we see from \eqref{r0} that $u$ solves
\bel{r1}
\left\{ \begin{array}{ll} -\Delta u = \tilde{f} & \mbox{in}\ \Om, \\ \pd_\nu u = g & \mbox{on}\ \Gamma, \end{array} \right. 
\ee
where $\tilde{f}:= f - q u \in H^k(\Om)$ and $g:= - \al u \in H^{k+\frac{1}{2}}(\Ga)$. Therefore, we have $u \in H^{k+1}(\Om)$ by 
\cite[Theorem 2.5.1.1]{G}.
\end{proof}

\subsection{Improved regularity of the eigenfunctions}
Let $\phi_n$ denote the $n$-th eigenfunction of the Schr\"odinger-Robin operator $A$. Then $-\Delta \phi_n=\lambda_n \phi_n$ in the distributional sense and $\partial_\nu\phi_n +\alpha\phi_n=0$ in $H^{-1/2}(\Gamma)$.
Suppose that $q$ and $\alpha$ fulfill either \eqref{c1b} or \eqref{c2b}. We get by applying Lemma \ref{lm-er} with $f=-\la_n \phi_n \in H^1(\Om)$ that $\phi_n \in H^3(\Om)$, provided that 
the boundary $\Ga$ is $\cC^{2,1}$, $q \in W^{1,\infty}(\Om)$ and $\al \in \cC^{1,r}(\Ga)$ for some $r \in \left( \frac{1}{2}, 1\right)$. Similarly, using that $\phi_n \in H^2(\Om)$, we get from Lemma \ref{lm-er} where $f=-\la_n \phi_n \in H^2(\Om)$ that $\phi_n \in H^4(\Om)$ whenever 
$\Ga$ is $\cC^{3,1}$, $q \in W^{2,\infty}(\Om)$ and $\al \in \cC^{2,r}(\Ga)$. By iterating, we obtain the following statement.

\begin{lemma}
\label{lm-ire}
For $k \in \N$, let $\Ga$ be of class $\cC^{k+1,1}$, $q \in W^{k,\infty}(\Om)$ and $\al \in \cC^{k,r}(\Ga)$, for some $r \in \left( \frac{1}{2}, 1\right)$. Furthermore, assume that $q$ and $\alpha$ fulfill either \eqref{c1b} or \eqref{c2b}. 
Then, we have $\phi_n \in H^{k+2}(\Om)$.
\end{lemma}

\end{document}